\documentclass[letterpaper, 10 pt, conference]{ieeeconf}  

\usepackage[T1]{fontenc}
\usepackage{bm}
\usepackage{amsmath}
\usepackage{algorithm}
\usepackage{algorithmic}
\usepackage{amsfonts}
\usepackage{mathtools}
\usepackage{xcolor}
\newtheorem{theorem}{Theorem}
\usepackage{booktabs} 
\usepackage{xcolor}
\usepackage{subcaption}

\usepackage{pdfpages}

\IEEEoverridecommandlockouts                              

\title{\LARGE \bf
Free-Horizon Newton Method for Nonlinear Optimal Control
}

\author{Gyeongrok Ha$^{1}$, Kenji Fujimoto$^{1}$, and Ichiro Maruta$^{1}$
\thanks{*This work was supported by JSPS KAKENHI Grant Numbers JP23K20946 and JP24K00908.}
\thanks{$^{1}$G. Ha, K. Fujimoto, and I. Maruta
 are with Department of Aeronautics and Astronautics, Graduate
 School of Engineering, Kyoto University, Kyoto, 615-8540, Japan
        {\tt\small ha.rok.77s@st.kyoto-u.ac.jp}}%
}

\begin{document}

\includepdf[pages=1]{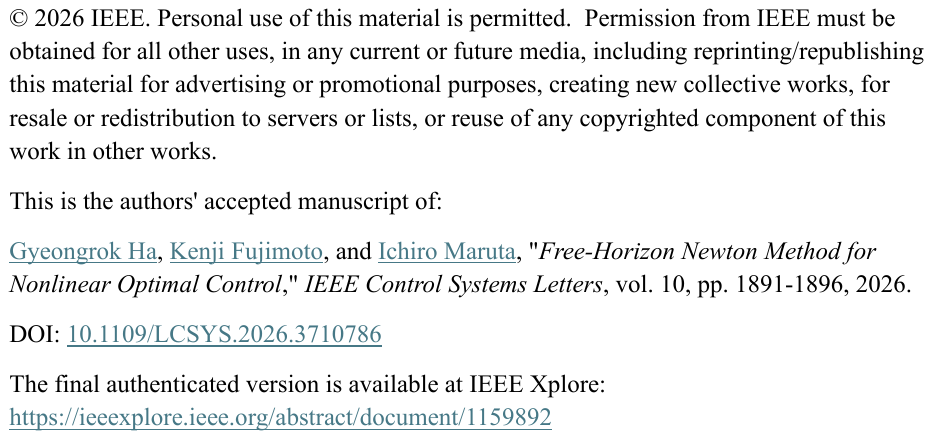}

\maketitle
\thispagestyle{empty}
\pagestyle{empty}

\begin{abstract}
This paper presents a novel trajectory optimization method for nonlinear $\ell^1$-optimal control problems in which the control horizon is treated as a free variable, allowing both the control duration and the $\ell^1$-norm of the control input to be evaluated. Unlike traditional approaches that minimize the $\ell^1$-norm of the input under a fixed control horizon, the proposed method treats both the control input and control horizon as design variables, enabling joint optimization of energy and temporal efficiency. The method extends a Newton-based algorithm to accommodate this objective, leveraging gradient information with respect to both variables. This formulation enables one to find high-fidelity optimal trajectories without pre-specifying the control horizon. The effectiveness and robustness of the method are demonstrated through numerical simulations, in which the proposed algorithm recovers the analytical solution to the spacecraft Hohmann transfer, and determines an optimal control horizon for an Earth--Moon transfer.

\end{abstract}

\section{INTRODUCTION}
Optimal trajectory design plays a central role in modern space mission planning, where both fuel efficiency and mission duration are critical design criteria \cite{larson1999space}. These problems are expressed as infinite or free-horizon optimal control problems, where both the control sequence and control horizon are treated as optimization variables. Methods to solve such problems are broadly categorized into direct and indirect approaches. Direct methods transcribe the problem into nonlinear programming formulations, where the free terminal time is commonly handled through time normalization, collocation, or time-dilation techniques \cite{garg2011pseudospectral,infeld2007design,elango2025continuous}, 
while indirect methods derive optimality conditions of the trajectory via Pontryagin’s maximum principle \cite{frego2014numerical,liberzon2011calculus}.

This paper proposes a direct method based on a Newton algorithm for solving nonlinear optimal control problems with an \(\ell^1\)-norm objective function analogous to the total control effort along a trajectory. A key feature of the proposed approach is that each iteration is formulated as an \(\ell^1\)-minimization problem, directly solving for a minimum-\(\ell^1\)-norm correction that satisfies the linearized terminal constraints. Consequently, the sparsity-inducing structure of the original minimum-fuel objective is preserved throughout the optimization process, making the method particularly well suited for spacecraft trajectory optimization and other applications where impulsive or bang-bang control structures are desirable \cite{polyak2020sparse,hamada2022ℓ}. In contrast to conventional Newton methods that require the control horizon to be fixed in advance, the proposed approach treats both the control horizon and the control input sequence as free optimization variables. This joint formulation enables the computation of trajectories that are simultaneously optimal with respect to both control horizon and control inputs. The proposed algorithm also incorporates an optimality verification step based on the Karush-Kuhn-Tucker (KKT) conditions to assess the first-order optimality of the obtained solution \cite{nocedal1999numerical}.

While general-purpose approaches such as continuous-time successive convexification (CT-SCVX) provide flexible frameworks for handling nonlinear dynamics, constraints, and free-horizon problems \cite{elango2025continuous}, they do not explicitly compute sparse control corrections through pure $\ell^1$-minimization subproblems at each iteration. In contrast, the proposed method directly minimizes a first-order model of the original objective while simultaneously accounting for variations in the control horizon. As a result, each subproblem retains the structure of the continuous-time minimum-fuel problem without introducing additional local time-scaling variables that alter the geometry of the $\ell^1$ objective. These properties make the proposed method particularly well matched to trajectory optimization problems whose optimal solutions are inherently sparse, while maintaining substantially lower computational complexity than more general nonlinear programming frameworks.

The structure of this paper is organized as follows. Section II introduces the trajectory optimization problem addressed in this work and reviews related research in the field. Section III presents the proposed algorithm, an extension of previous Newton-based algorithms that simultaneously optimizes both the control input sequence and the transfer time. It also includes an analysis of the algorithm's convergence properties and the optimality of the computed solution. Section IV demonstrates the effectiveness of the method through numerical simulations of spacecraft trajectory optimization. Finally, Section V concludes the paper and highlights its key contributions.

\section{Background and Preparations}
\label{sec:Background}
This section introduces relevant literature that the present study builds upon \cite{hamada2022ℓ}. We formulate the $\ell^1$-optimal control problem, which can be solved by a Newton-based algorithm that iteratively updates the control sequence.

\subsection{The $\ell^1$-Optimization Problem}
Consider the following continuous dynamics of an optimal control problem,
\begin{align}
    \dot{\bm{x}}(t)=\bm{f}_\mathrm{c}(\bm{x}(t),\widetilde{\bm{u}}(t)), \quad t \in [0,T]\label{continuousdynamics}
\end{align}
where $\bm{x}\in\mathbb{R}^n$ and $\widetilde{\bm{u}}\in\mathbb{R}^m$ represent the state and input. Using an appropriate discretization scheme, such as a fourth-order Runge-Kutta method (RK4), and assuming a standard zero-order hold (ZOH) on the control input over each sampling interval, i.e., $\widetilde{\bm{u}}(t) = \widetilde{\bm{u}}[j]$ for $t \in [j\tau, (j+1)\tau)$, the following discrete-time model is obtained:
\begin{align}
    \bm{x}[j+1]&=\bm{f}(\bm{x}[j],\widetilde{\bm{u}}[j];\tau),\,\,\,j=0,\dots,N-1\label{problem1}
\end{align}
Here, $\bm{x}[j]\in\mathbb{R}^n$ and $\widetilde{\bm{u}}[j]\in\mathbb{R}^m$ denote the state and input at time step $j$, respectively. The time step $\tau$, which is the elapsed time between discretization points, is defined by the equidistant discretization of the control horizon $T$ as
\begin{align}
    \tau=T/N.\label{taudef}
\end{align}
We also define $\bm{u}\in\mathbb{R}^{Nm}$, which denotes the aggregated control input over all the time steps $j=0,\dots,N-1$. 
\begin{align}
    \bm{u}&=[\widetilde{\bm{u}}[0]^\top,\,\widetilde{\bm{u}}[1]^\top,\dots\,\widetilde{\bm{u}}[N-1]^\top]^\top\label{problem2}
\end{align}
We can see that from \eqref{problem1} and \eqref{problem2}, the terminal state $\bm{x}[N]$ is uniquely determined from the initial state $\bm{x}[0]$ and the control sequence $\bm{u}$. Thus, the $\ell^1$-minimization problem can be expressed as follows:
\begin{align}
    \underset{\bm{u}}{\text{minimize}}\hspace{0.5cm}&\|\bm{u}\|_1\,,\label{l1normobj}\\
    \text{s.t.}\hspace{0.5cm}&\text{(\ref{problem1})}\,,\bm{x}[0]=\bm{x}_0\,,\,\bm{x}[N]=\bm{x}_\text{f}\,.\label{problem4}
\end{align}

The control effort is given by \eqref{l1normobj}, where $\|\cdot\|_1$ denotes the $\ell^1$-norm of the control sequence vector. By solving the above nonlinear program, the control sequence $\bm{u}$ that minimizes the $\ell^1$-norm under the control horizon $T$ is obtained.

\subsection{$\ell^1$-Optimal Newton Method}
To solve the $\ell^1$-optimal control problem, Hamada et al. proposed an iterative scheme to find a locally $\ell^1$-optimal solution that satisfies boundary conditions \cite{hamada2022ℓ}.

To apply the $\ell^1$-optimal Newton method, the following constraint is first defined.
\begin{align}
    \bm{g}(\bm{u}) \coloneqq \bm{x}[N](\bm{u}) - \bm{x}_\text{f}\in \mathbb{R}^n\label{constraint}
\end{align}
Here, $\bm{x}[N]$ is computed from the successive relation \eqref{problem1} using the control sequence $\bm{u}$. We can see that when $\bm{g}(\bm{u}^*)=\bm{0}$, the boundary conditions \eqref{problem4} are satisfied, and the input $\bm{u}=\bm{u}^*$ is feasible.

Given an initial guess $\bm{u}_0$, the algorithm updates the inputs $\bm{u}_k$ according to the following update rule.
\begin{align}
    \bm{u}_{k+1} = \bm{u}_k - \bm{w}_k \label{l1optimal1}
\end{align}
The update step $\bm{w}_k$ at step $k$ is determined by solving the following linear program:
\begin{align}
    \underset{\bm{w}_k}{\text{minimize}} & \hspace{0.5cm} \|\bm{u}_k-\bm{w}_k\|_1 \label{l1optimal2} \,.\\
    \text{s.t.} & \hspace{0.5cm} \mathcal{J}_{\bm{g}}(\bm{u}_k) \bm{w}_k = \bm{g}(\bm{u}_k)\,,\label{l1optimal3}\\
    & \hspace{0.5cm} \|\bm{w}_k\|_1\leq\alpha_1\max(\|\bm{g}(\bm{u}_k)\|_1,\epsilon)\,.\label{l1optimal4}
\end{align}
Here, $\mathcal{J}_{\bm{g}}(\bm{u}_k)\coloneqq\frac{\partial\bm{g}}{\partial\bm{u}}(\bm{u}_k)$ is the Jacobian matrix of $\bm{g}(\bm{u})$ evaluated at $\bm{u}=\bm{u}_k$. The parameters \(\alpha_1>0\) and \(\epsilon>0\) control the convergence behavior of the algorithm and must satisfy specific conditions to ensure quadratic convergence of the error \(\|\bm{g}(\bm{u}_k)\|_1\). See \cite{hamada2022ℓ} for details.

\section{Proposed Method}
\label{sec:proposed}

In this section, we propose the free-horizon Newton method by modifying the $\ell^1$-optimal Newton method to accommodate a free horizon.
A proof of quadratic convergence and a method to verify the first-order optimality conditions are also provided.

\subsection{The Free-Horizon Model}
We begin by considering a continuous-time optimal control problem that motivates the discrete-time formulation introduced to accommodate the Newton-based algorithm. Let the system dynamics be described by \eqref{continuousdynamics} with $t\in[0,T]$, and where the terminal time $T>0$ is treated as a free variable. To quantify the control effort over a variable time horizon, we consider the following performance index:
\begin{align}
    J_\mathrm{c}(\widetilde{\bm{u}}(t),T) =\int_{0}^{T}\|\widetilde{\bm{u}}(t)\|_{1}\,\mathrm{d}t.
\end{align}

Under the ZOH discretization, the continuous-time cost reduces to
\begin{align}
    \int_{0}^{T}\|\widetilde{\bm{u}}(t)\|_1\,\mathrm{d}t
    \;=\;
    \tau\sum_{j=0}^{N-1}\|\widetilde{\bm{u}}[j]\|_1
    =\frac{T}{N}\|\bm{u}\|_1.\label{objfuncdiscretization}
\end{align}
This discretization naturally leads to a dependence of the objective function on the terminal time $T$. With this interpretation in mind, we next formulate the discrete-time free-horizon optimization problem.

To enable choosing an arbitrary control horizon for the Newton method, we treat the horizon $T$ as a free variable. To accommodate this additional parameter, we define the following extended decision vector.
\begin{align}
    \bm{U}\coloneqq(\bm{u}^\top,T)^\top\in\mathbb{R}^{Nm+1}
\end{align}
Within this model, the control horizon is no longer fixed, and the objective function must be adjusted accordingly. The control effort of a given trajectory with temporal considerations is given by the discretization \eqref{objfuncdiscretization}. Since the discretization number $N$ is fixed,
the objective function can be written as simply
\begin{align}
    J(\bm{U})=T\|\bm{u}\|_1\,. \label{nonconvexobj}
\end{align}
Thus, the free-horizon optimal control problem is formulated as the following static optimization problem:
\begin{align}
    \underset{\bm{U}}{\text{minimize}}\hspace{0.5cm}&J(\bm{U})\,,\label{propproblem1}\\
    \text{s.t.}\hspace{0.5cm}&\eqref{problem1}\,,\eqref{taudef}\,,\bm{x}[0]=\bm{x}_0\,,\,\bm{x}[N]=\bm{x}_\text{f}\,.\label{propproblem4}
\end{align}

Note that this minimization objective is nonconvex, and cannot be directly minimized using linear programming methods such as \eqref{l1optimal2}--\eqref{l1optimal4}. To address this, we apply a first-order approximation of \eqref{nonconvexobj}. Specifically, we consider small deviations of the extended input $\bm{U}\rightarrow \bm{U}+\Delta\bm{U}$ where $\bm{u}\rightarrow \bm{u}+\Delta\bm{u}$ and $T\rightarrow T+\Delta T$, and neglect second-order terms. Under this approximation, the objective becomes
\begin{align}
    J(\bm{U}+\Delta\bm{U})= T\|\bm{u}+\Delta\bm{u}\|_1+\Delta T\|\bm{u}\|_1 + o(\| \Delta \bm{U} \|)\,\label{linobjective}
\end{align}
where the term $o(\| \Delta \bm{U} \|)$ is ignored when solving the linear program.

Lastly, using the same convention as \eqref{constraint}, we express the terminal state constraint in \eqref{propproblem4} as a function of $\bm{U}$.
\begin{align}
    \bm{G}(\bm{U}) \coloneqq \bm{x}[N](\bm{U}) - \bm{x}_\text{f}\in \mathbb{R}^n \label{proposedconstraint}
\end{align}

The Jacobian matrix of this newly defined constraint is defined as follows.
\begin{align}
    \mathcal{J}_{\bm{G}}(\bm{U})\coloneqq\left(\frac{\partial\bm{G}}{\partial\bm{u}},\frac{\partial\bm{G}}{\partial T}\right)\in\mathbb{R}^{n\times(Nm+1)}
\end{align}
The elements of this matrix can be computed analytically from the recursive relation \eqref{problem1} and the chain rule. We are now ready to apply the Newton scheme on the optimization problem \eqref{propproblem1}--\eqref{propproblem4}.

\subsection{The Update Rule}

Here, we describe the Newton-based update rule to generate a trajectory that is optimal in terms of both the input sequence and the control horizon. The extended input is updated by the following update rule, initialized by an initial guess $\bm{U}_0=(\bm{u}_0^\top,T_0)^\top$.
\begin{align}
    \bm{U}_{k+1}=\bm{U}_{k}-\bm{W}_{k}\label{proposed1}
\end{align}
Here, the extended update term $\bm{W}_k$ is composed of the following.
\begin{align}
    \bm{W}_k = (\bm{w}_k^\top,S_k)^\top\in\mathbb{R}^{Nm+1}\label{proposed2}
\end{align}
At step $k$, the update terms $\bm{w}_k$ and $S_k$ are determined by solving the following linear program:
\begin{align}
    \underset{\bm{w}_k,S_k}{\text{minimize}} & \hspace{0.3cm} T_k\|\bm{u}_k-\bm{w}_k\|_1-S_k\|\bm{u}_k\|_1\,,\label{proposed3}\\
    \text{s.t.} & \hspace{0.3cm} \mathcal{J}_{\bm{G}}(\bm{U}_k) \bm{W}_k = \bm{G}(\bm{U}_k)\,,\label{proposed4}\\
    & \hspace{0.3cm} \|\bm{w}_k\|_1\leq\alpha_1\max(\|\bm{G}(\bm{U}_k)\|_1,\sigma_k)\,,\label{proposed5}\\
    & \hspace{0.3cm} |S_k|\leq\min (\bar{S},\alpha_2\max(\|\bm{G}(\bm{U}_k)\|_1,\sigma_k))\,.\label{proposed6}
\end{align}

The parameters $\alpha_1,\alpha_2,\sigma_k,\bar{S}>0$ are algorithm parameters that control the step size of the control update and the horizon update, and are selected to ensure convergence of the algorithm, as demonstrated in the next subsection. The novelty of the proposed algorithm compared to \cite{hamada2022ℓ} lies primarily in \eqref{proposed3}, \eqref{proposed4}, and \eqref{proposed6}. The objective \eqref{proposed3} is now \eqref{linobjective}, with higher order terms ignored to ensure that the subproblem is a linear program. Note that the first term of \eqref{proposed3} corresponds to a direct $\ell^1$-minimization of the next input, which naturally favors sparse corrections at each subproblem \cite{hamada2022ℓ}. \eqref{proposed4} permits updates to the horizon, with \eqref{proposed5} and \eqref{proposed6} establishing a trust region. Particular attention must be paid to the choice of \(\bar{S}\), which sets an upper bound on the step size taken by \(T_k\). This constraint is essential to prevent instability during the update process, since \(\bm{G}(\bm{U})\) exhibits high sensitivity to the value of \(T\). The step-dependent parameter \(\sigma_k\) is designed to allow the algorithm to continue generating updates \(\bm{W}_k\) of significant size, even when the error \(\|\bm{G}(\bm{U}_k)\|_1\) is close to zero.

\subsection{Convergence Analysis}
Here, we provide a proof of convergence of the update rule \eqref{proposed1}--\eqref{proposed6}, given a few regularity assumptions. 
\begin{theorem}
    Consider the case $\sigma_k=0$. Let $\mathcal{U}$ be the domain of $\bm{U}$, where the following are assumed to hold:
\begin{itemize}
  \item[\textbf{A1.}] \((\partial\bm{G}/\partial\bm{u})\) is full row rank for all \(\bm{U} \in \mathcal{U}\).
  \item[\textbf{A2.}] $ \|\left(\partial\bm{G}/\partial\bm{u}\right)^{+} \|_1\leq \alpha_1,\ \forall\, \bm{U} \in \mathcal{U}.$
  \item[\textbf{A3.}] \(\exists \lambda>0,\) s.t. \(\forall\,\bm{U}', \bm{U}'' \in \mathcal{U}\) with $|T'-T''|\leq\bar{S},$\\ $\|\mathcal{J}_{\bm{G}}(\bm{U}') - \mathcal{J}_{\bm{G}}(\bm{U}'')\|_1 \leq \lambda \|\bm{U}' - \bm{U}''\|_1.$
  \item[\textbf{A4.}]
   \(\lambda (\alpha_1+\alpha_2)^2\|G(\bm{U}
   _0)\|_1 / 2 < 1.\)
\end{itemize}
Then, the error $\|\bm{G}(\bm{U}_k)\|_1$ achieves quadratic convergence.
\end{theorem}

\begin{proof}
    Consider the case $S_k=0$. The pseudoinverse $\left(\partial\bm{G}/\partial\bm{u}\right)^{+}$ provides the minimum $\ell^2$-norm solution $\bm{w}_\text{min,2}=\left(\partial\bm{G}/\partial\bm{u}\right)^{+}\bm{G}(\bm{U}_k)$ to the constraint \eqref{proposed4}, which satisfies the following relation with the minimum $\ell^1$-norm solution $\bm{w}_\text{min,1}$.
    \begin{align}
        \|\bm{w}_\text{min,1}\|_1\leq\|\bm{w}_\text{min,2}\|_1\label{minnormsolutions}
    \end{align}
    Due to the consistency of the $\ell^1$-norm, \textbf{A1} and \eqref{minnormsolutions}, we obtain
    \begin{align}
        \|\bm{w}_\text{min,1}\|_1\leq\|\left(\partial\bm{G}/\partial\bm{u}\right)^{+}\|_1\|\bm{G}(\bm{U}_k)\|_1.
    \end{align}
    Thus, we can see that a solution to the linear program \eqref{proposed3}--\eqref{proposed6} always exists when \textbf{A2} holds. When considering the case $S_k\neq0$, an update \eqref{proposed2} always exists as the constraint \eqref{proposed4} becomes more relaxed with the addition of $\left(\partial\bm{G}/\partial T\right)S_k$. From the above discussion, we can see \textbf{A1} and \textbf{A2} ensure that \eqref{proposed4} always has a solution under \eqref{proposed5}--\eqref{proposed6}.

    Next, we use the following relation.
    \begin{align}
        \bm{G}(\bm{U}-\bm{W})=\bm{G}(\bm{U})-\int_0^1\mathcal{J}_{\bm{G}}(\bm{U}-\xi\bm{W})\bm{W}\,\text{d}\xi
    \end{align}
    From the above, it follows that
    \begin{align}
        &\bm{G}(\bm{U}_{k+1})=\bm{G}(\bm{U}_k)-\mathcal{J}_{\bm{G}}(\bm{U}_k)\bm{W}_k\nonumber\\
        &\qquad+\int_0^1(\mathcal{J}_{\bm{G}}(\bm{U}_k)-\mathcal{J}_{\bm{G}}(\bm{U}_k-\xi\bm{W}_k))\bm{W}_k\,\text{d}\xi.
    \end{align}
    Using the triangle inequality for normed spaces, the consistency of the $\ell^1$-norm, the constraint \eqref{proposed4}, and \textbf{A3}, we obtain the following.
    \begin{align}
        \|\bm{G}(\bm{U}_{k+1})\|_1&=\bigg\|\int_0^1(\mathcal{J}_{\bm{G}}(\bm{U}_k)-\mathcal{J}_{\bm{G}}(\bm{U}_k-\xi\bm{W}_k))\bm{W}_k\,\text{d}\xi\bigg\|_1\nonumber\\
        &\leq\frac{\lambda}{2}\|\bm{W}_k\|_1^2\label{g_w_relation}
    \end{align}
    The relation $\|\bm{W}_k\|_1=\|\bm{w}_k\|_1+|S_k|$, step size constraints \eqref{proposed5}--\eqref{proposed6}, and \eqref{g_w_relation} lead to
    \begin{align}
        \|\bm{G}(\bm{U}_{k+1})\|_1\leq\frac{\lambda(\alpha_1+\alpha_2)^2}{2}\|\bm{G}(\bm{U}_k)\|_1^2,\label{quadconvergence}
    \end{align}
    which alongside \textbf{A4}, shows the quadratic convergence of $\|\bm{G}(\bm{U}_k)\|_1$.
\end{proof}

While the preceding theorem addresses the case of \(\sigma_k = 0\), the proposed method is intended to operate with a nonzero value of \(\sigma_k\). The following theorem establishes the result for the case \(\sigma_k > 0\).

\begin{theorem}
Consider the case $\sigma_k =\sigma$, where $\sigma>0$ is a constant. Assume that \textbf{A1}-\textbf{A4} in Theorem 3.1 are satisfied on $\bm{U}\in\mathcal{U}$, along with an additional assumption:
\begin{itemize}
   \item[\textbf{A5.}]
   \(\sigma<2/[\lambda(\alpha_1+\alpha_2)^2].\)
\end{itemize}
Then, there exists an update step $\bar{k}> 0$ such that $\|\bm{G}(\bm{U}_k)\|_1\leq\sigma$ holds for all $k\geq\bar{k}$.
\end{theorem}
\begin{proof}
    While $\|\bm{G}(\bm{U}_k)\|_1>\sigma$, the error decreases monotonically as shown in Theorem 3.1. This continues until a certain update step \(k = \bar{k}\), at which point the error drops below the threshold \(\sigma\). Once $\|\bm{G}(\bm{U}_k)\|_1\leq\sigma$ holds, i.e., $k\geq\bar{k}$, \eqref{proposed5}, \eqref{proposed6}, and \eqref{g_w_relation} leads to the relation
    \begin{align}
        \|\bm{G}(\bm{U}_{k+1})\|_1\leq\frac{\lambda(\alpha_1+\alpha_2)^2\sigma^2}{2}.\label{sigmabarconvergence}
    \end{align}
    Owing to \textbf{A5}, the above inequality leads to the conclusion
    \begin{align}
        \|\bm{G}(\bm{U}_{k})\|_1\leq\sigma, \forall k\geq\bar{k}.
    \end{align}
    Overall, the convergence of the error \(\|\bm{G}(\bm{U}_k)\|_1\) proceeds in two distinct phases. In the first phase of $k<\bar{k}$, where \(\|\bm{G}(\bm{U}_k)\|_1 > \sigma\), the error decreases monotonically. In the second phase of $k\geq\bar{k}$, the inequality \(\|\bm{G}(\bm{U}_k)\|_1 \leq \sigma\) remains satisfied for all subsequent iterations, up to the termination of the algorithm.
\end{proof}

Note that \textbf{A1}--\textbf{A3} are reasonable, as the control landscape is assumed to be sufficiently smooth despite nonlinearities. \textbf{A4} may not hold when the initial guess $\bm{U}_0$ is inadequate. \textbf{A5} is a reasonable assumption, as the step size parameter \(\sigma\) can be chosen arbitrarily small. However, this introduces a trade-off: if \(\sigma\) is set too small, it may needlessly inflate the number of steps the algorithm takes to reach an \(\ell^1\)-optimal solution. To mitigate this issue, one may employ a step-dependent parameter \(\sigma_k\) that gradually decreases to zero. A simple example is the following adaptive step size rule, which includes a switching mechanism.
\begin{align}
    \sigma_k = \sigma \min(1, \gamma^{k - k_\text{s}})\label{adaptivestep}
\end{align}
Here, the parameter \(0 < \gamma < 1\) controls the decay rate of the step size, while \(k_\text{s}\) denotes the switching iteration.

\subsection{Optimality Analysis}


To ensure that the decision vector generated by the proposed method is locally optimal, the algorithm incorporates a verification step based on the KKT conditions \cite{nocedal1999numerical}. Using the KKT residual to assess stationarity, iterations are continued until both the constraint violations and the stationarity residual satisfy their prescribed tolerances, thereby preventing premature termination before a locally optimal solution is obtained.

The optimization problem being solved is of the following form:
\begin{align}
    \underset{\bm{U}}{\text{minimize}}\hspace{0.5cm}&F(\bm{U})=T\|\bm{u}\|_1\,,\label{kkt_obj}\\
    \text{s.t.}\hspace{0.5cm}&\bm{G}(\bm{U})=\bm{x}[N](\bm{U})-\bm{x}_\text{f}=\bm{0}\,,\label{kkt_eq}\\
    &\bm{H}(\bm{U})=\begin{bmatrix}
        |\bm{u}|-u_\text{lim}\bm{1}_{Nm}\\
        T_\text{min}-T\\
        T-T_\text{max}
    \end{bmatrix}\leq\bm{0}\,.\label{kkt_ineq}
\end{align}
Here, $|\cdot|$ and $\bm{1}_{Nm}$ denote the elementwise absolute value operator $|\bm{u}|=(|u_1|,\dots,|u_{Nm}|)^\top$ and a vector of ones of dimension $Nm$. The parameters $u_\mathrm{lim},T_\mathrm{min},T_\mathrm{max}$ provide bounds to the input sequence and the control horizon, and can be readily accommodated by adding the constraints \\$\,|\bm{u}_k-\bm{w}_k|\leq u_\text{lim}\bm{1}_{Nm}$ and $T_\text{min}\leq T_k-S_k\leq T_\text{max}$
to the linear program \eqref{proposed3}--\eqref{proposed6}.

Of the four original KKT conditions, only the stationarity, dual feasibility, and complementary slackness conditions need to be checked for. They are listed as follows:
\begin{align}
    &\bm{0} \in \partial F(\bm{U}) + \partial \bm{G}(\bm{U})^\top \bm{\lambda} + \partial \bm{H}(\bm{U})^\top \bm{\mu}\,,\label{stationarity}\\
    &\qquad\quad\bm{\mu} \geq \bm{0}\,,\quad\bm{\mu}^\top\bm{H}(\bm{U}) = 0\,.\label{dualfeas_compslack}
\end{align}
In the above equations, $\bm{\lambda}\in\mathbb{R}^n$ and $\bm{\mu}\in\mathbb{R}^{Nm+2}$ are Lagrange multipliers corresponding to the equality and inequality constraints of \eqref{kkt_eq} and \eqref{kkt_ineq}, respectively. The conditions of primal feasibility are ignored, as the check for optimality is initiated only when they are fulfilled during the update process of \eqref{proposed1}--\eqref{proposed6}. While these conditions can be shown to hold at convergence via the correspondence between the convex subproblems \eqref{proposed3}--\eqref{proposed6} and the original problem \eqref{kkt_obj}--\eqref{kkt_ineq}, similarly to \cite{yamamoto2025sparse}, extending the proof to the free-horizon case requires additional assumptions on the horizon update and is omitted due to space limitations.

In the presence of nonsmoothness and subdifferential mappings, we assess stationarity by minimizing the residual of the generalized KKT conditions, as supported by variational analysis frameworks and stationarity diagnostics in nonsmooth optimization \cite{rockafellar317jb}. The residual
\begin{align}
    \bm{R}(\bm{U})
    \;\coloneqq\;
    \partial F^\star(\bm{U})
    + \partial \bm{G}(\bm{U})^\top \bm{\lambda}^\star
    + \partial \bm{H}(\bm{U})^\top \bm{\mu}^\star,
\end{align}
is defined as the result of minimizing the Euclidean distance between the right hand side of \eqref{stationarity} and the zero vector $\bm{0}$, where the set $\{\partial F^\star,\bm{\lambda}^\star,\bm{\mu}^\star\}$ is given as the minimizer of the optimization problem
\begin{align}
    \underset{\partial F,\bm{\mu},\bm{\lambda}}{\text{minimize}}\hspace{0.2cm}\|\partial F + \partial \bm{G}^\top \bm{\lambda} + \partial \bm{H}^\top \bm{\mu}\|_2\,,\hspace{0.2cm}\text{s.t.}\hspace{0.2cm}&\eqref{dualfeas_compslack}.\nonumber
\end{align}
When an extended decision vector \(\bm{U}\) is provided, the algorithm for verifying the KKT conditions proceeds by attempting to solve the above minimization problem.

The specific minimization problem being solved in practice is expressed as the following quadratic programming:
\begin{align}
    \underset{\partial F,\bm{\mu},\bm{\lambda}}{\text{minimize}}\hspace{0.5cm}&\|\partial F + \partial \bm{G}^\top \bm{\lambda} + \partial \bm{H}^\top \bm{\mu}\|_2\,,\label{KKTresidual0}\\
    \text{s.t.}\hspace{0.5cm}&-T\leq(\partial F)_{1:Nm}\leq T\,,\label{KKTresidual1}\\
    &(\partial F)_{Nm+1} = \|\bm{u}\|_1\,,\label{KKTresidual2}\\
    &S_1(\partial F)_{1:Nm} = T\operatorname{sgn}(\bm{u})\,,\label{KKTresidual3}\\
    &\bm{\mu}\geq0\,,\,\,S_2\bm{\mu}=\bm{0}\,.\label{KKTresidual4}
\end{align}
Here, the operator $\operatorname{sgn}(\cdot)$ is used as an elementwise sign function. The notation $(\cdot)_{i}$ denotes the $i$-th element of a vector, while $(\cdot)_{i:j}$ denotes the vector containing elements from $i$ to $j$. Constraints \eqref{KKTresidual1}, \eqref{KKTresidual3} and \eqref{KKTresidual4} express that the equality or inequality holds for each element.

The matrices $\partial\bm{G}\in\mathbb{R}^{n\times(Nm+1)}$ and $\partial\bm{H}\in\mathbb{R}^{(Nm+2)\times(Nm+1)}$ are computed from their definitions of \eqref{kkt_eq} and \eqref{kkt_ineq} as follows:
\begin{align}
    \partial\bm{G}=\mathcal{J}_{\bm{G}}(\bm{U})
\end{align}
\begin{align}
    (\partial\bm{H})_{ij}=\begin{cases}
        &\operatorname{sgn}((\bm{u})_i) \hspace{0.5cm}(i=j, i\leq Nm)\\
        &-1 \hspace{0.4cm} (i=Nm+1,j=Nm+1)  \\
        &1 \hspace{0.69cm}  (i=Nm+2,j=Nm+1)    \\
        &0 \hspace{0.69cm}\text{otherwise}
    \end{cases}
\end{align}
Meanwhile, the matrices $S_1\in\mathbb{R}^{Nm\times Nm}$ and $S_2\in\mathbb{R}^{(Nm+2)\times (Nm+2)}$ are defined as follows.
\begin{align}
    (S_1)_{ij}&=\begin{cases}
        &1 \hspace{0.6cm}(i=j,|(\bm{u})_i|\neq 0)\\
        &0 \hspace{0.6cm}\text{otherwise}
    \end{cases}\\
    (S_2)_{ij}&=\begin{cases}
        &1 \hspace{0.6cm}(i=j\leq Nm,|(\bm{u})_i|\neq u_\text{lim})\\
        &1 \hspace{0.6cm}(i=j=Nm+1,|T-T_\text{min}|
        \neq 0)\\
        &1 \hspace{0.6cm}(i=j=Nm+2,|T-T_\text{max}|
        \neq 0)\\
        &0 \hspace{0.6cm}\text{otherwise}
    \end{cases}
\end{align}
These matrices enforce dual feasibility and complementary slackness while minimizing the stationarity residual. Constraint \eqref{KKTresidual1} bounds $\partial F$ to within $[-T,T]$, while \eqref{KKTresidual3} sets elements corresponding to nonzero $\bm{u}$ to $\pm T$. The remaining element is determined by \eqref{KKTresidual2}, and \eqref{KKTresidual4} enforces \eqref{dualfeas_compslack}.



A summary of the proposed method is provided in \textbf{Algorithm 1}, where $\epsilon_g$ and $\epsilon_R$
are tolerances for $\|\bm{G}(\bm{U}_k)\|_1$ and $\|\bm{R}_k\|_2$, respectively.

\begin{algorithm}
  \caption{: Free-Horizon Newton Method}
  \begin{algorithmic}
    \STATE \textbf{Input:} $N, \bm{x}_0, \bm{x}_\text{f}, \bm{U}_0, k_\text{lim}, \sigma_k,\alpha_1,\alpha_2,\bar{S},\epsilon_g, \epsilon_R$
    \STATE $k \leftarrow 0$
    \WHILE{$k\leq k_\text{lim}$}
        \STATE $\bm{W}_k\leftarrow$ Solve \eqref{proposed2}--\eqref{proposed6}
        \STATE $\bm{U}_{k+1} \leftarrow \bm{U}_k - \bm{W}_k,\,\,k \leftarrow k + 1$
    \IF{$\|\bm{G}(\bm{U}_k)\|_1 \leq \epsilon_g$}
        \STATE $\bm{R}_k\leftarrow$ Verify KKT conditions by \eqref{KKTresidual0}--\eqref{KKTresidual4}
        \IF{$\|\bm{R}_k\|_2\leq\epsilon_R$}
        \STATE \textbf{break}
    \ENDIF
    \ENDIF
    \ENDWHILE
    \RETURN $\bm{U}_k=(\bm{u}_k^\top,T_k)^\top$
  \end{algorithmic}
\end{algorithm}

\section{Numerical Examples}
\label{sec:numerical}


This section presents two numerical experiments to validate the proposed method. Experiment 1 considers the Hohmann transfer in the two-body problem, whose analytical optimal solution \cite{prussing1992simple} enables direct comparison with the proposed algorithm, while Experiment 2 considers an Earth--Moon transfer in the circular restricted three-body problem (CR3BP) \cite{szebehely2012theory} to demonstrate its effectiveness on a more complex and practically relevant problem.

In the two-body problem, the spacecraft dynamics are given by
\begin{equation}
    \frac{\mathrm{d}\bm{x}}{\mathrm{d}t} = \bm{f}_{\mathrm{c},\text{TBP}}(\bm{x}(t), \widetilde{\bm{u}}(t)) \coloneqq \begin{bmatrix}
        x_3 \\
        x_4 \\
        -\frac{\partial \phi_1}{\partial x_1} + u_x \\
        -\frac{\partial \phi_1}{\partial x_2} + u_y
    \end{bmatrix},
    \label{twobodydynamics}
\end{equation}
where $\bm{x} = [x_1, x_2, x_3, x_4]^\top$ denotes position and velocity, and $\phi_1(x_1,x_2)$ is the gravitational potential
\begin{align}
    \phi_1(x_1,x_2) = -\frac{GM}{\sqrt{x_1^2 + x_2^2}}.
\end{align}

In contrast, in the circular restricted three-body problem (CR3BP) expressed in a rotating frame, the dynamics become
\begin{equation}
    \frac{\mathrm{d}\bm{x}}{\mathrm{d}t}  = \bm{f}_{\mathrm{c},\text{CR3BP}}(\bm{x}(t),\widetilde{\bm{u}}(t))\coloneqq \begin{bmatrix}
        x_3 \\
        x_4 \\
        2x_4 + \frac{\partial \phi_2}{\partial x_1} + u_x \\
        -2x_3 + \frac{\partial \phi_2}{\partial x_2} + u_y
    \end{bmatrix},
    \label{cr3bpdynamics}
\end{equation}
where $\phi_2(x_1,x_2)$ is the pseudo-potential, defined using $\mu \in (0,1)$, the mass ratio of the two primary bodies,
\begin{align}
    \phi_2(x_1,x_2) &= \frac{1 - \mu}{\sqrt{(x_1 + \mu)^2 + x_2^2}} + \frac{\mu}{\sqrt{(x_1 - 1 + \mu)^2 + x_2^2}} \nonumber \\
    &\quad + \frac{1}{2}(x_1^2 + x_2^2).
\end{align}

All simulations are implemented in MATLAB software \cite{MATLAB2023} on an AMD Ryzen 9 9950X processor. The discrete dynamics \eqref{problem1}--\eqref{problem2} are obtained by applying RK4 to \eqref{twobodydynamics} and \eqref{cr3bpdynamics} using $N=1000$ steps over a variable time horizon. The optimization problems in \eqref{proposed3}--\eqref{proposed6} and \eqref{KKTresidual0}--\eqref{KKTresidual4} are solved using Gurobi \cite{gurobi}. The experimental settings are summarized in Table I. In Experiment 1, the initial and final states lie on circular orbits of radii 1 and 3, respectively, while in Experiment 2 they correspond to circular parking orbits $200\,\mathrm{km}$ above Earth and $2000\,\mathrm{km}$ above the Moon, respectively. For both experiments, the maximum iteration count and convergence tolerances are set to $k_{\mathrm{lim}}=100$, $\epsilon_g=10^{-3}$, and $\epsilon_R=10^{-3}$. For comparison, Experiment 1 is also solved using sequential quadratic programming (SQP) and an interior-point method (IPM) implemented in \texttt{fmincon} as representative general-purpose nonlinear programming solvers, with analytic derivatives provided to ensure a fair comparison.

\begin{table}[h]
    \centering
    \caption{Experiment settings}
    \begin{tabular}{c|c|c}
        \toprule
        Variable & Experiment 1 & Experiment 2 \\
        \midrule
        $GM$ or $\mu$ & $GM=4\pi^2$ & $\mu=0.0122$\\
        $\bm{x}_0$ & $[0,1,2\pi,0]^\top$ & $[-0.0242,-0.0121,5.38,-5.38]^\top$\\
        $\bm{x}_\text{f}$ & $[-3,0,0,\frac{2\sqrt{3}}{3}\pi]^\top$ & $[0.988,-0.00972,1.12,0]^\top$\\
        $\bm{u}_0$ & $\bm{0}_{Nm}$ & $[2420,0,\dots,0]^\top$\\
        $T_0$ & $0.75$ & $1.4$\\
        $\sigma_k$ & $0.1$ & \eqref{adaptivestep}, $\sigma = 1.0 , \gamma = 0.9, k_\text{s} = 25$\\
        $\alpha_1$ & $4.0\times10^3$ & $4.0\times10^3$\\
        $\alpha_2$ & $5.0\times10^{-1}$ & $1.0\times10^{-2}$ \\
        $\bar{S}$ &$5.0\times10^{-1}$ &$1.0\times10^{-2}$\\
        \bottomrule
    \end{tabular}
\end{table}

\begin{table}[t]
    \centering
    \caption{Comparison of solutions for the Hohmann transfer problem in Experiment 1}
    \begin{tabular}{c|c|c|c|c}
        \toprule
         & Proposed & SQP & IPM & Analytic \\
        \midrule
        $\|\bm{G}\|_1$ & $8.72\cdot 10^{-6}$ & $5.40\cdot 10^{-7}$ &  $7.99\cdot 10^{-6}$ & $0$  \\
        $\|\bm{R}\|_2$ & $3.41\cdot 10^{-4}$ & $1.40\cdot 10^{3}$ & $1.70\cdot 10^{3}$ & $0$ \\
        $\Delta V$ & $2.47$ & $2.68$ & $2.84$ & $2.47$ \\
        $T$ & $1.66$ & $1.74$ & $1.73$ & $1.66$ \\
        Iterations & 12 & 1000 & 1000 & - \\
        Runtime [$\mathrm{s}$] & $3.79$ & $213$ & $544$ & - \\
        \bottomrule
    \end{tabular}
\end{table}


The results are shown in Figures 1--3, with each subfigure corresponding to Experiments 1 and 2. In Experiment 1, the proposed method converged in $k_\text{end}=12$, yielding a trajectory identical to the analytically optimal Hohmann transfer with two impulsive maneuvers. The numerically obtained solution also closely matches the analytical one, as shown in Table II, while the small KKT residual further confirms optimality. Note that the constraint residual in Figure 3 is not monotonically decreasing because the assumptions made in Section III.C were initially violated.

Although both the SQP and IPM benchmarks of Experiment 1 readily produced feasible solutions, neither converged to the optimal solution despite requiring substantially more iterations and computational time than the proposed method. Both algorithms were terminated after 1000 iterations without satisfying the prescribed convergence criteria. At termination, the resulting control inputs remained far from sparse (Fig. 2(a)), and the corresponding trajectories deviated significantly from the Hohmann transfer solution (Fig. 1(a)).

In Experiment 2, convergence was achieved at $k_{\text{end}} = 56$ with $20.7 \,\mathrm{s}$ of runtime. The adaptive step size \eqref{adaptivestep} enabled gradual horizon extension and residual reduction, as can be seen in Figure 3-(b). The final horizon was $T = 1.70$, with a trajectory cost of $\Delta V = 4.10$. The terminal residuals, $\|\bm{G}\|_2 = 9.95 \times 10^{-4}$ and $\|\bm{R}\|_1 = 2.01 \times 10^{-5}$, further confirm the feasibility and local optimality of the resulting trajectory.

\section{Conclusion}\label{sec:conclusion}
In this paper, a free-horizon Newton method was developed for solving nonlinear $\ell^1$-optimal control problems in which both the control input and the control horizon are treated as optimization variables. By reformulating the cost as the product of the control horizon and the $\ell^1$-norm of the control input, the method effectively captures control effort under a variable horizon. The algorithm employs a first-order approximation of the objective and incorporates KKT-based optimality verification to ensure convergence and solution quality.

Numerical simulations of Hohmann and Earth--Moon transfers demonstrate that the proposed method is capable of computing high-fidelity solutions efficiently. These results confirm the method's potential for trajectory design in optimal control applications where the transfer time is not predetermined.

\section*{Acknowledgment}
This work was supported by JSPS KAKENHI Grant Numbers JP23K20946 and JP24K00908.

\begin{figure}[h!]
    \centering
    \begin{subfigure}{0.475\linewidth}
        \centering
        \includegraphics[width=\linewidth]{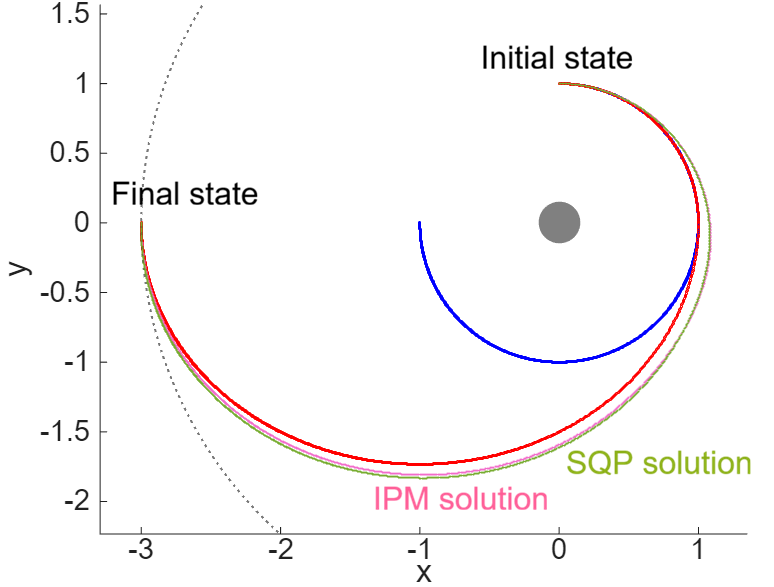}
        \caption{Experiment 1}
    \end{subfigure}
    \hfill
    \begin{subfigure}{0.495\linewidth}
        \centering
        \includegraphics[width=\linewidth]{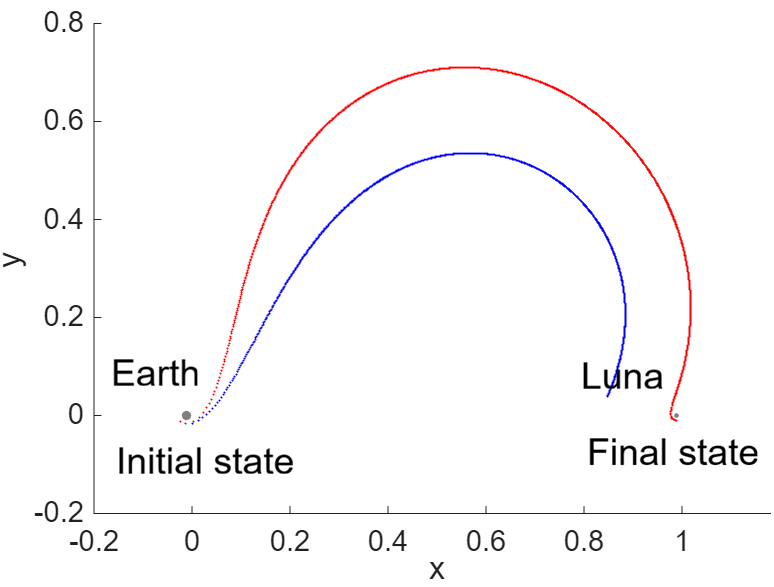}
        \caption{Experiment 2}
    \end{subfigure}
    \caption{Numerical trajectories. Blue: initial guess $\bm{U}_0$. Red: proposed solution. SQP/IPM shown in (a).}
\end{figure}

\begin{figure}[h!]
    \centering
    \begin{subfigure}{0.48\linewidth}
        \centering
        \includegraphics[width=\linewidth]{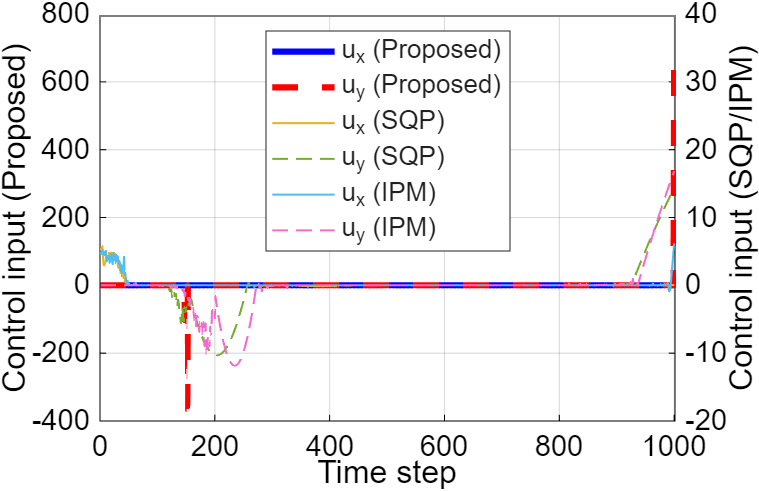}
        \caption{Experiment 1}
    \end{subfigure}
    \hfill
    \begin{subfigure}{0.49\linewidth}
        \centering
        \includegraphics[width=\linewidth]{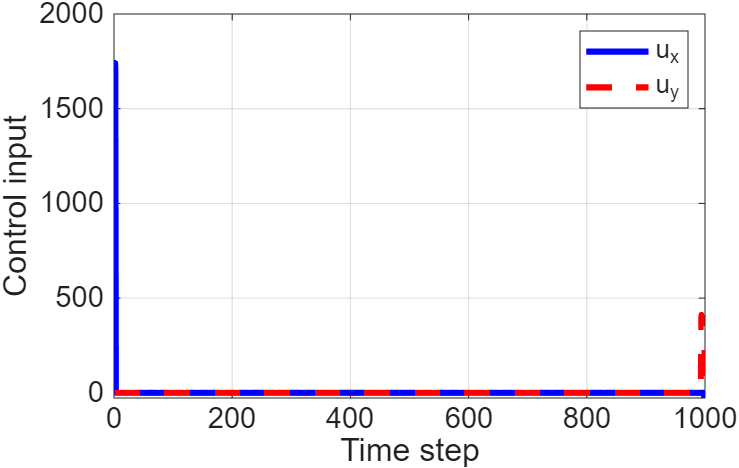}
        \caption{Experiment 2}
    \end{subfigure}
    \caption{The control inputs $u_x$ and $u_y$ obtained from each experiment.}
\end{figure}

\begin{figure}[h!]
    \centering
    \begin{subfigure}{0.48\linewidth}
        \centering
        \includegraphics[width=\linewidth]{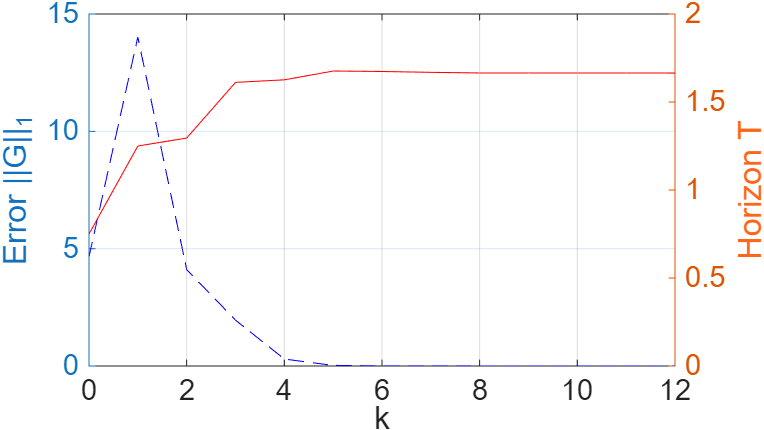}
        \caption{Experiment 1}
    \end{subfigure}
    \hfill
    \begin{subfigure}{0.49\linewidth}
        \centering
        \includegraphics[width=\linewidth]{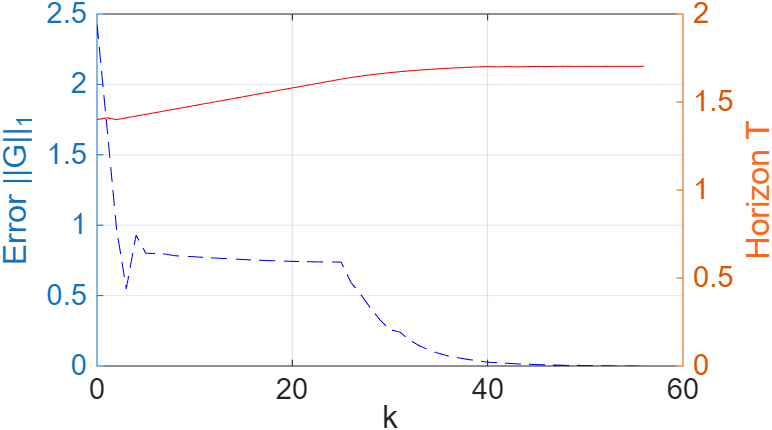}
        \caption{Experiment 2}
    \end{subfigure}
    \caption{Evolution of the constraint residual $\|\bm{G}\|_1$ and horizon $T$ throughout the experiments.}
\end{figure}

\addtolength{\textheight}{-12cm}   





\end{document}